\documentclass[12pt,reqno]{article}
\usepackage{graphicx}
\usepackage{amsmath}
\usepackage{amsthm}
\usepackage{latexsym,bm}
\usepackage{amssymb}
\usepackage{indentfirst}
\usepackage{pstricks}
\usepackage{pst-plot}
\usepackage{pst-eps}
\usepackage{pst-grad}
\newtheorem{thm}{Theorem}[section]
\newtheorem{cor}[thm]{Corollary}
\newtheorem{lem}[thm]{Lemma}
\newtheorem{prop}[thm]{Proposition}

\numberwithin{equation}{section}

\begin{document}

\title{{  \bf  Finiteness of  prescribed   fibers of   local biholomorphisms: a geometric approach
}}

\author{{\bf F. Xavier}}
\author{Xiaoyang Chen \footnote{The first author is supported in part by NSF DMS-1209387.} \ and Frederico Xavier }
 \date{}
\maketitle

\begin{abstract}
\noindent  Let  $X$ be a Stein manifold of complex dimension at least two, $F:X  \to \mathbb C^n$   a local biholomorphism,  and $q\in F(X)$.  In this paper we formulate sufficient conditions, involving only objects  naturally associated to $q$, in order  for  the fiber over $q$ 
 to be finite. Assume that  $F^{-1}(l)$ is $1$-connected  for the  generic complex line $l$ containing   $q$, and   $F^{-1}(l)$ has    finitely many
  components whenever  $l$ is    an exceptional line through  $q$. Using  arguments from topology and  differential geometry,   we    establish    a sharp estimate on the size of $F^{-1}(q)$.    It follows that  for $n\geq 2$
  a    local biholomorphism   of  $X$ onto $\mathbb C^n$   is  invertible if and only if    the pull-back of every complex line is   $1$-connected.
 \end{abstract}
\section{Introduction.}
\vskip10pt

The  theory  of holomorphic  maps  $F:X \to \mathbb C^n$,  where $X$ is a connected non-compact complex manifold of dimension $n\geq 2$,    has  been the object of  much research   over the years.
As expected,   most  of these works  make extensive  use  of  various  analytical tools, a prime example being Nevanlinna theory.

In   this paper, by contrast,   we use  only    topological and geometric arguments to address   the  problem   of estimating the size  of  a  fiber, specified in advance,   of a given local biholomorphism:

\begin{thm} \label{line}Let $X$ be a connected non-compact complex manifold of dimension at least two,  $F: X\to \mathbb C^n$ a local biholomorphism, and $q$ a point in the image of $F$. Assume that
\vskip3pt
\noindent a) $X$ carries a  complete K$\ddot{ a}$hler metric of  negative holomorphic sectional curvature.
\vskip3pt
\noindent b) The pre-image of the generic complex line containing   $q$ is $1$-connected.
\vskip3pt
\noindent c)  The pre-image of  any  exceptional  complex line through  $q$ has  finitely many
components.
\vskip3pt
\noindent Then,  the fiber $F^{-1}(q)$ is finite. In fact,
\begin{eqnarray} \label {sharp} \#F^{-1}(q)\leq \min_{V\subset \mathbb C^n} \max_{l\in V} \{\# \text {components of } \; F^{-1}(l)\}, \end{eqnarray}
where the minimum is taken over all  complex $2$-planes $V$ containing $q$, and the maximum is taken over all  complex lines $l$  in $V$ that pass through $q$.
\end{thm}

The topological hypotheses b) and c) in Theorem \ref{line} are essential for the validity of the theorem, and  the estimate  (\ref{sharp}) is sharp (see Section $3$). The  necessity for  the geometric  hypothesis a), on the other hand,   seems to be  a  subtler  matter,    and remains an unsettled question.

The  word \it generic \rm  in the statement of  Theorem \ref{line}  is, of course,  to be understood in the usual sense.  For the sake of clarity,  and to establish notation, we  recall its meaning.  After a translation by $-q$, the set of complex lines in $\mathbb C^n$ that pass through $q$ is naturally identified with the complex projective space $\mathbb C \mathbb P^{n-1}$.  A property  $\cal P$ about such  lines  is  satisfied  generically   if there exists a complex hypersurface $\Sigma \subset \mathbb C \mathbb P^{n-1}$  such that $\mathcal P$ holds for all   elements  of $ \mathbb C \mathbb P^{n-1}-\Sigma$.  Lines associated to $\Sigma$ are deemed non-generic or \it exceptional\rm. If  $n=2$,
 $\cal P$ is satisfied generically if it holds for all lines  arising from the complement of a finite  subset of $\mathbb C \mathbb P^1$.

It should be pointed out  that  the estimate on the size of the fiber $F^{-1}(q)$ involves only objects naturally associated to $q$.  We also observe that  when  $n=2$     the actual  finite number of exceptional   lines in  Theorem \ref{line} is not relevant for the estimate on the size of the fiber.
In this case, (\ref{sharp}) reduces to
\begin{eqnarray} \label{sharp'}\#F^{-1}(q)\leq \max_l \{\# \text {components of } \; F^{-1}(l)\},\end{eqnarray}
where a)-c) are in force,  and the maximum is taken over all complex lines $l$  containing   $q$.

The theorem   is new  even   in the classical case  $X=\mathbb C^n$, $n\geq 2$. In fact, it holds   if $X$ is a Stein manifold.  Indeed, an explicit construction of  complete K$\ddot{\text a}$hler metrics in $\mathbb C^N$ with  negative sectional curvature, hence  with negative holomorphic sectional   curvature,  can be found in  \cite{S} (see also \cite{K}, \cite{SZ}).  Being  Stein, $X$ can be  properly  embedded as  a complex submanifold of some $\mathbb C^N$. Endowing the latter with a metric as above, it follows from (\cite{KN}, p.164, 176) that the induced metric on $X$   satisfies  a) in Theorem \ref{line}.

\begin{cor} Let $X$ be a connected Stein manifold of  complex  dimension at least two,  $F: X\to \mathbb C^n$ a local biholomorphism, and $q$ a point in the image of $F$. If   $F^{-1}(l)$ is $1$-connected  for the  generic complex line $l$ containing   $q$, and   $F^{-1}(l)$ has    finitely many
  components whenever  $l$ is    an exceptional line through  $q$, then    $F^{-1}(q)$  is finite. \end{cor}

The next  result    is a direct consequence of   estimate (\ref{sharp}):

\begin{cor}  \label{2} Let $X$ be a connected complex $n$-manifold that carries a complete  K$\ddot{a}$hler  metric of  negative  holomorphic sectional curvature, and   $F:X \to \mathbb C^n$   a local biholomorphism,  $n\geq 2$.   If the pre-image under $F$ of every affine complex line  that intersects $F(X)$ is  $1$-connected, then $F$ is injective.
\end{cor}

\begin{cor}  A  surjective local biholomorphism    $F:\mathbb C^n \to \mathbb C^n$, $n\geq 2$,  is invertible  if  and only if   the pre-image of every  complex line  is  a $1$-connected set.
\end{cor}

The proof of Theorem \ref{line}   represents   a substantial  elaboration of    the basic  idea   that motivated  the main result  in \cite{NX1}:
\vskip5pt
\noindent \it A local biholomorphism $F: Y\to \mathbb C^n$, $n\geq 2$, is injective if the pre-image of every complex line that intersects $F(Y)$ is conformal to  a connected rational curve.  \rm
\vskip5pt
(As communicated to us by S. Nollet, in the special case of certain algebraic maps the above result follows from the famous  Bend-and-Break Lemma of S. Mori.) However,   we stress that   the techniques  and results in   \cite{NX1}  are quite different  from the ones in  the present paper.

The  starting point  of these investigations is   the naive   observation  that  a  map  into $\mathbb C^n$ that has  discrete fibers is  injective if and only if the pre-image of  any given point is a connected set.

The next logical step  consists in trying  to infer  injectivity (\cite{B},\cite {NX1}) or,   in the present case, the finiteness of fibers,  from the  requirement  that the pre-images of one-dimensional objets (such as  complex lines), rather than points, are  connected.

Thus far   injectivity   can be achieved  if $X$ carries the right geometry, but only  under the stronger assumption that  the pull-backs of  {\em all}  complex lines that intersect the image are not only connected, but actually $1$-connected (Corollary \ref{2}).

In what follows we  elaborate on the relationship between   differential geometry  and   the   problem of estimating the size of  a   fiber of a local biholomorphism. Along the way,  we provide  a rough outline for   the proof of Theorem \ref{line}.

The pre-image of a generic complex line is a $1$-connected two-dimensional  real surface that   is properly embedded in $X$ (hence complete  relative to  the induced  metric). Furthermore, it  has  negative curvature (more generally, the holomorphic  sectional curvature of a K$\ddot{\text a}$hler submanifold is not greater than that of the ambient manifold  (\cite{KN}, p.164, 176)).

Thus, the pre-image of such a complex line is a Cartan-Hadamard surface, and therefore any two distinct points in it can be joined by a unique geodesic.

After reducing the proof of Theorem \ref{line} to the (complex) two-dimensional case,  in Theorem \ref{est},  one can take advantage of the above mentioned uniqueness  of geodesics and    construct a continuous nowhere vanishing section $s$  of the restriction of the tautological line bundle $L$ to the complement of the  finite set in $\mathbb C \mathbb P^1$  that corresponds to the exceptional lines whose pre-images are  not $1$-connected.

One then proceeds to show -- and this is the heart of the matter --,  that if the estimate  (\ref{sharp}) does not hold then each one of the  singularities of $s$ must have  index zero.

The arguments  employed to  establish the vanishing of the local indices   involve  once again  the use of geodesics, this time  to  deform continuously the section $s$ into another one whose  local index at the singularity in question is manifestly zero.    By the Poincar\'e-Hopf theorem,  the Euler number  $e(L)$ of $L$ is  the sum of the local indices of all the singularities of the section $s$, and so  it would be  zero,  contradicting  $e(L)=-1$.

The present paper is part of a larger program whose aim is to study global invertibility and related questions in the complex-analytic setting, using geometric, analytic and topological tools
(\cite{FX},  \cite {NX1}, \cite{NTX}, \cite{X3}).

\section{Proof of  the main result. }
\vskip10pt

Theorem \ref{line}  admits  a  slightly stronger formulation, which we will  state  and prove in this section.  Replacing $F$ by $F-F(p)$, where $F(p)=q$,  we may assume    $q=0$.
\begin{thm}\label{est} Let $X$ be a connected complex $n$-manifold that carries a complete  K$\ddot{a}$hler  metric of  negative  holomorphic sectional curvature,
$n\geq 2$, $F:X \to \mathbb C^n$  a local biholomorphism,   $0\in F(X)$.
Assume the existence of a    subspace $V\subset \mathbb C^n$,  $\text{dim} V=2$,  and  of an  integer $k\geq1$ such that:
\vskip5pt
\noindent a) $F^{-1}(V)$ is connected.
\vskip3pt
\noindent b) For  all but finitely many complex lines  $l\subset V$ that pass through $0$,   $F^{-1}(l)$ is $1$-connected.  If $l$ is  any of these  exceptional lines,  $F^{-1}(l)$ has at most $k$ components.
\vskip5pt
\noindent
Then the fiber $F^{-1}(0)$ has at most $k$ elements.
\end{thm}

We now proceed to show how Theorem \ref{line} follows from Theorem \ref{est}. Let $V\subset \mathbb C^n$ be a two-dimensional subspace, as in the statement of Theorem \ref{line}, with the property that
its image in projective space intersects $\Sigma$ (notation as in the Introduction) on a finite set. Since there are only finitely many one-dimensional subspaces $l\subset V$ for which $F^{-1}(l)$ is not $1$-connected,  hypothesis c) of Theorem \ref{line} implies the existence of $k\geq 1$ satisfying b) in Theorem \ref{est}. It is now clear that Theorem \ref{est} implies Theorem \ref{line} if one can establish that $F^{-1}(V)$ is connected.

To this end, let us argue by contradiction and write $F^{-1}(V)=U_1\cup U_2$, where $U_1, U_2$ are non-empty open sets   such that  $U_1\cap U_2=\emptyset$.

The space of lines in $V$ that pass through $q=0$   can be  naturally identified with the projective plane
$\mathbb C \mathbb P^1$.  If $\pi$ is the natural projection,  let   $E=\pi (V-\{0\})\cap \Sigma\subset \mathbb C \mathbb P^1$  be the finite (perhaps empty) set  corresponding to the  lines  whose pre-images are not $1$-connected.  As the pre-image $F^{-1}(l)$ of a complex line $l$ is connected when
 $\pi(l)\in \mathbb C \mathbb P^1-E$,   one has
\begin{eqnarray*}\{\pi(l)\in \mathbb C \mathbb P^1-E: F^{-1}(l) \cap U_j\neq \emptyset \}=\{\pi(l)\in \mathbb C \mathbb P^1-E: F^{-1}(l) \subset  U_j \}.\end{eqnarray*} For $j=1,2$, these  sets  are non-empty, open, disjoint, and their union is   $\mathbb C \mathbb P^1-E$. Since $E$ is finite, $\mathbb C \mathbb P^1-E$ remains connected, and we have  a contradiction. Thus,  as claimed, $F^{-1}(V)$ is connected, finishing the proof that Theorem \ref{line} follows from Theorem \ref{est}. \qed

\vskip10pt

The proof of  Theorem \ref{est} itself is fairly involved,  and will take the remainder of this section. We start by explaining  why it suffices to consider only the case $\text{dim}_{\mathbb  C} X = 2$.

Notice that, by b) of Theorem \ref{est}, $F^{-1}(0)$ is  contained   in  a single component of
$F^{-1}(V)$. Since the latter is connected, by a), one sees that
\begin{eqnarray} \label {restriction} F^{-1}(0)= (F|V)^{-1}(0). \end{eqnarray}

The metric induced from $X$ makes  $F^{-1}(V)$ into   a  connected  complete K$\ddot{\text a}$hler  surface  of  negative  holomorphic sectional curvature (\cite{KN}, p.164, 176).  In view of (\ref{restriction}), we may  therefore assume in Theorem \ref{est}   that  $X=F^{-1}(V)$,   and $F:X\to \mathbb C^2$ is a local biholomorphism satisfying b) of Theorem \ref{est}. In this context, $V=\mathbb C^2$ so that a) is trivially satisfied. Furthermore, $\Sigma \subset \mathbb C \mathbb P^1$ is a finite set.

After the above reduction to the case $\text{dim}_{\mathbb  C} X = 2$, we  argue  by contradiction and assume that $F:X\to \mathbb C^2$ satisfies
 $\#F^{-1}(0)> k \geq 1$.

 In particular,  one can choose distinct points   $p_1, p_2 \in F^{-1}(0)$.
For any $l \in \mathbb C \mathbb P^1- \Sigma$, since $F^{-1}(l)$ is properly embedded, it follows from (\cite{KN}, p. 164, 176),  that the holomorphic curve $F^{-1}(l)$ can be viewed as a complete connected real surface of  negative curvature,
relative  to the  metric induced  from the complete metric of negative holomorphic sectional curvature on $X$. On the other hand,
by assumption, $F^{-1}(l)$ is simply-connected.
Hence, at every point of  $F^{-1}(l)$,  the exponential map is a diffeomorphism by the  classical  Cartan-Hadamard theorem (\cite {Pet},  162-163). In particular, any two points in $F^{-1}(l)$ can  be joined by a \it unique \rm geodesic.

Again, for $l \in \mathbb C \mathbb P^1- \Sigma$,  let   $\omega_l:[0,1]\rightarrow F^{-1}(l)$ be the (unique) geodesic such that $\omega_l(0)=p_1, \omega _l(1)=p_2.$
It is easy to see that   $\omega_l'(0)$ depends  continuously on $l$.  This is so because   geodesics converge to geodesics in the $C^2$
topology,
together with the fact that  there is only one  geodesic in  $F^{-1}(l)$ joining  $p_1$ to $p_2$.

Consider the tautological line bundle over $\mathbb C \mathbb P^1$,  whose total space is
\begin{eqnarray} L=\{{(l,v)|l\in \mathbb C \mathbb P^1, v \in l}\}\subset \mathbb C \mathbb P^1 \times \mathbb C^2.\end{eqnarray}
It is a well-known fact that the Euler number of $L$ satisfies $e(L)=-1\neq 0$.

The map
\begin{eqnarray}
s: \mathbb C \mathbb P^1- \Sigma \rightarrow L, \;\;\;  l \mapsto dF(\omega_l'(0)),
\end{eqnarray}
will play a central role in our proof. Notice that, since $dF(p)$ is everywhere non-singular,    $s$ is a nowhere-zero continuous section of the restriction of the tautological line bundle to
the complement in $\mathbb C \mathbb P^1$ of the finite set   $\Sigma$.  Alternatively, we can view $s$ as a  section of $L$ with  (finitely  many) singularities at the  points of $\Sigma$.

Our strategy to prove Theorem \ref{est} consists in showing that, in the presence of the condition $\#F^{-1}(0)> k \geq 1$, the
local index of $s$ at each one of these singularities  is zero.  Then,  by Poincar$\acute{e}-$Hopf index theorem (\cite{Bot}, p. 123-124),
the Euler number of $L$  would be  zero, contradicting  $e(L)=-1$. Thus, as claimed in the theorem, one would have $\#F^{-1}(0)\leq k$.

For a   fixed  complex line $l_\alpha \in \Sigma$,  choose a  sufficient small neighborhood $D^2$ of $l_\alpha$ in $\mathbb C \mathbb P^1$
such that $D^2$ is diffeomorphic to an open ball in $\mathbb{R}^2$,  and the restriction of $L$ to $D^2$ is a trivial bundle.
The associated circle bundle to $L$, denoted by $SL$, is again a trivial bundle when restricted to $D^2$.

As an open subset of $\mathbb C \mathbb P^1$, $D^2$ is also oriented.
Choose the orientation on the circle $S^1$ in such a way that the isomorphism of  circle bundles
$SL_{|D^2}\simeq D^2 \times S^1$ is
orientation-preserving, where $D^2 \times S^1$ is given the product orientation.

Then,  the local index of the section $s$ at  the singularity $l_\alpha$ is the degree of the composite map:

\begin{eqnarray}
\Phi: \partial \bar{D^2} \rightarrow SL_{|\bar{D^2}}\simeq \bar{D^2} \times S^1 \rightarrow S^1, \;
l\mapsto \frac{s(l)}{\|s(l)\|} \mapsto \pi (\frac{s(l)}{\|s(l)\|}),
\end{eqnarray}
where $\bar{D^2}$ is the closure of $D^2$, $\partial \bar{D^2}$ is the boundary of $\bar{D^2}$ and $\pi: \bar{D^2} \times S^1 \rightarrow S^1$ is the projection.
Here, $\|s(l)\|^2=g_0(s(l),s(l))$,  where $g_0$ is the standard inner product on $\mathbb{C}^2$.
\\
\par On the other hand, since we are assuming $\#F^{-1}(0)> k \geq 1$,  and $F^{-1}(l_\alpha)$ has at most $k$ connected components,
there must exist  some  component $M^2$ of $F^{-1}(l_\alpha)$ that  contains at least two different points $p_3, p_4 \in F^{-1}(0)$.

Notice that $p_1, p_2, p_3,p_4$ are not necessarily pairwise  distinct, but $p_1\neq p_2$,  and $p_3\neq p_4$. One should also observe that, although $p_1$ and $p_2$ have been chosen at the very beginning of the proof, the choices of $p_3=p_3^{\alpha}$ and $p_4=p_4^{\alpha}$ depend, in principle, on the particular exceptional line $l_{\alpha}$.

Despite the fact that  the components of  $F^{-1}(l_\alpha)$  are  not assumed to be simply-connected in Theorem \ref{est}, one still has uniqueness of geodesics joining $p_3$ to $p_4$. This is the content of the next lemma.  Without it, Theorem \ref{est} would be slightly weaker. The conclusion  would still be the same, but condition c)  would  have to be  strengthened  to require  that  each connected component of  the pre-images of the exceptional lines is simply-connected.

 \vskip10pt
\begin{lem} \label{first}
There is a unique geodesic $\upsilon_{l_\alpha}:[0,1]\rightarrow M^2\subset F^{-1}(l_{\alpha})$ joining $p_3$ to $p_4$.
\end{lem}

 \begin{proof} By the completeness of $M^2$, there is at least one such geodesic. Suppose that  $\gamma_1, \gamma_2: [0,1]\rightarrow M^2$ are two geodesics with  $\gamma_1(0)=\gamma_2(0)=p_3$ and
 $\gamma_2(1)=\gamma_2(1)=p_4$.  Let $U\subset M^2$ be a  bounded open set containing $\gamma_1([0,1])$ and $\gamma_2([0,1])$.
 For each $x\in U$, consider the  geodesics $\Gamma:(-\epsilon, \epsilon)\to X$ with $\Gamma(0)=x$ and $\Gamma'(0)\perp T_xM^2$.
 Let $\Sigma_x$ be the local surface generated by these local geodesics $\Gamma$. In particular, $\Sigma_x$ intersects $M^2$ transversally at $x$.
 It follows that for every complex line $l \in \mathbb{CP}^1$ sufficiently close to $l_{\alpha}$,  $\Sigma_x \cap F^{-1}(l)$ is a single point. In  fact, for every
 $x\in \overline U$, there is a neighborhood $U_x$ of $x$ in $M^2$, and a neighborhood $W_x$ of $l_{\alpha}$ in $\mathbb C \mathbb P^1$ such that, for every
 $y\in U_x$ and $l\in W_x$, one has  that $\Sigma_y \cap F^{-1}(l)$ is a single point.

 Since $\overline U$ is compact, it can be covered by finitely many open sets $U_x$. Letting $W$ be the intersection of the (finitely many) corresponding sets $W_x$, we can then define a map $\pi_{U,l}$, for each $l\in W$,  as follows:

\begin{eqnarray}
\pi_{U,l}:U \to F^{-1}(l),  \; x \mapsto \Sigma_x \cap F^{-1}(l).
 \end{eqnarray}

 Furthermore, shrinking $W$ if necessary,  $\pi_{U,l}$ is an immersion. We also remark that  the map $\pi_{U,l_{\alpha}}: U \to F^{-1}(l_{\alpha})$ is just the inclusion.
 For $l \in W$, consider $F^{-1}(l)$ with its induced metric from $X$, and denote by $g(l)$ the  metric on $U$ induced by the immersion $\pi_{U,l}$.

Denote by $\exp^{g(l)}_{p_3}$ the exponential map of  the metric $g(l)$, based at $p_3$.
Since $F$ is a local biholomorphism,  for any finite positive integer $k$ the submanifold $F^{-1}(l)$ converges to $F^{-1}(l_{\alpha})$,  $C^k$-uniformly  over  compact subsets of $X$,
as the complex line $l$ tends to $l_{\alpha}$.  In particular, taking $W$ smaller if necessary,
we can assume that  the domain of  $\exp^{g(l)}_{p_3}$ contains $\gamma_1'(0)$ and $\gamma_2'(0)$ for all $l\in W$.

Given a curve $(-\epsilon, \epsilon)\to W$, $t\mapsto l(t)$,  $l(0)=l_{\alpha}$, we want to find  a smooth curve $t\mapsto \xi(t)$
in $T_{p_3}X$, such that
\begin{eqnarray} \label{EXP} H(t,\xi(t)):=\exp^{g(l(t))}_{p_3} (\xi(t))=p_4, \;\;\; \xi(0)=\gamma_1'(0).
\end{eqnarray}
That such a curve $\xi(t)$ exists follows from the implicit function theorem. Indeed, one only needs to check that $\frac{\partial H}{\partial \xi}|_{t=0}= d \;exp^{g_{l_\alpha}}_{p_3}$ is invertible, and this is the case  as $M$ has  non-positive curvature.

 As a consequence of (\ref{EXP}),  for small enough $t\neq 0$ there is a geodesic in the Riemannian manifold $(U, g(t))$ joining $p_3$ to $p_4$, with initial vector  close to $\gamma_1'(0)$. Since the metric $g(t)$ is induced by the immersion $\pi_{U,l}:U\to F^{-1}(l)$,  and the points $p_3$ and $p_4$ are kept fixed by  $\pi_{U,l}:U\to F^{-1}(l)$, it follows that, likewise, there is a geodesic in $F^{-1}(l(t))$ that joins $p_3$ to $p_4$, and whose tangent vector  can be made as close to $\gamma_1'(0)$ as one wants by taking $t>0$ sufficiently small.

 The preceding discussion applies to  the geodesic $\gamma_2$ as well. Hence, for $t>0$ sufficiently close to zero, there are  geodesics in $F^{-1}(l(t))$ that join $p_3$ to $p_4$, and whose initial vectors   converge  to $\gamma_1'(0)$ and $\gamma_2'(0)$ in $T_{p_3}X$, respectively, as $t\to 0$.

 But  $F^{-1}(l)$  is $1$-connected for $l\neq l_{\alpha}$, and so  there is  only one geodesic in $F^{-1}(l)$ joining $p_3$ and $p_4$.  In particular,  from the above paragraph one must have
 $\gamma_1'(0)=\gamma_2'(0)$, so that the geodesics $\gamma_1$ and $\gamma_2$ are the same.
 This concludes the proof of Lemma \ref{first}.
 \end{proof}

By taking the disc $D^2$ used in the definition of $\Phi$ to be small enough, we can assume that
$D^2 \cap \Sigma =\{{l_\alpha}\}$. Then,  for any $l\in D^2$, $l \neq l_\alpha$, $F^{-1}(l)$ is
$1$-connected.

Hence,  $p_3, p_4 \in F^{-1}(l)$ and there is a unique geodesic
 $\upsilon_l:[0,1]\rightarrow F^{-1}(l)$ such that
$\upsilon_l(0)=p_3, \upsilon_l(1)=p_4$. Let $\Psi$ be the composite map

\begin{eqnarray}
\Psi: \partial \bar{D^2} \rightarrow SL_{|\bar{D^2}}\simeq \bar{D^2} \times S^1 \rightarrow S^1, \;\;\;
l\mapsto \frac{\sigma (l)}{\|\sigma (l)\|} \mapsto \pi (\frac{\sigma (l)}{\|\sigma (l)\|}),
\end{eqnarray}
where $\upsilon_l'(0)=\frac{d}{dt}\upsilon_l(t)|_{t=0}$, $\pi$ is the projection onto the second factor, and
\begin{eqnarray}
\sigma: \partial \bar{D^2} \rightarrow L, \;\;\; l \mapsto dF(\upsilon_l'(0)).
\end{eqnarray}

Notice the similarity  between  the definitions of the maps $\Phi$ and $\Psi$: in the first instance the geodesics join $p_1$ to $p_2$, while in the second they join  $p_3$ to $p_4$.

By Lemma \ref{first}, we can define  $\bar{\sigma}:  \bar{D^2} \rightarrow L$  by
\begin{eqnarray}
  l_\alpha \mapsto dF(\upsilon_{l_\alpha}'(0)), \;\ \text{and}\;\;
l \mapsto dF(\upsilon_l'(0)) \;\;\text{if}\;\; l \neq l_\alpha.
\end{eqnarray}
A simple argument using the uniqueness of $\upsilon_{l_{\alpha}}$ shows  that  $\bar{\sigma}$ is a continuous  extension of $\sigma$. Accordingly,
$\Psi$ extends to  the composite map

\begin{eqnarray}
\bar{\Psi}: \bar{D^2} \rightarrow SL_{|\bar{D^2}} \simeq \bar{D^2} \times S^1 \rightarrow S^1, \;
l\mapsto \frac{\bar{\sigma} (l)}{\|\bar{\sigma} (l)\|} \mapsto \pi (\frac{\bar{\sigma} (l)}{\|\bar{\sigma} (l)\|}).
\end{eqnarray}

\vskip10pt

\begin{lem} \label{metric} Denote by $g$  the metric of $X$. Then, there exists a compact set $K\subset X$ and disjoint neighborhoods $U_{\alpha}$ in $\mathbb C \mathbb P^1$ of  the finitely many exceptional lines $l_{\alpha}$  such that,  for every $\epsilon>0$,  there exists a metric $\hat g$ on $X$ satisfying
\vskip5pt
\noindent i) $\text{supp} \; (g-\hat g)\subset K$,  $\{p_1, p_2, p_3^{\alpha}, p_4^{\alpha}\}\cap K=\emptyset.$
\vskip3pt
\noindent ii) $||g-\hat g||_{C^2}\leq \epsilon. $
\vskip3pt
\noindent iii) For any $l\in U_{\alpha}$, no three points in the set  $\{p_1, p_2, p_3^{\alpha}, p_4^{\alpha}\}$ lie on the same geodesic of
$F^{-1}(l)$, relative to the metric induced from $\hat g$.
\end{lem}

\begin{proof} Since geodesics converge to geodesics,  once  iii) is established for the exceptional lines $l_{\alpha}$ themselves,  the existence of $U_{\alpha}$ follows.

Suppose, for instance, that  for some $l_{\alpha}$ fixed,  the oriented geodesic on $F^{-1}(l_{\alpha})$ joining $p_1$ and $p_2$ also contains $p_3=p_3^{\alpha}$, and $p_3>p_2$ (in the obvious sense).

Let $q$ be a point in the open geodesic segment $(p_2, p_3)$ that does not belong to the (discrete) set $F^{-1}(0)$.  Choose   a small geodesic ball
$B_{\epsilon}(q)$ in the ambient manifold $X$, and set  $A=B_{\epsilon}(q)\cap F^{-1}(l_{\alpha})$.

For $p\in \partial A$,  denote by  $\nu(p)$ the set of unit tangent vectors  at $p$ corresponding to the geodesics in $F^{-1}(l_{\alpha})$ joining $p$ to $p_3$.  Since the curvature is negative,  it follows from the arguments in Lemma \ref{first} that  the set $\nu(p)$  actually consists of a single point  if $\epsilon$ is small enough, and the map $\nu:A\to T\partial A, \;\;\; p \mapsto \nu(p)$,  is continuous.

It is now clear that there is  a metric $\tilde g$  on $F^{-1}(l_{\alpha})$   for which  $\tilde g-g|F^{-1}(l_{\alpha})$ is supported inside $A$,
$||\tilde g-g|F^{-1}(l_{\alpha})||_{C^2}$ is as small as desired and, furthermore, the following  holds:

\vskip8pt
\noindent ($\dagger$)  The   $\tilde g$-geodesic whose initial point is the first point in $\partial A$ of  the continuation of the $g$-geodesic $[p_1, p_2]$ hits the boundary $\partial A$ again at some  point  $p$ with a tangent vector that is not a multiple of    $\nu(p)$.
\vskip8pt
We omit the details that ($\dagger$)  can always be achieved, but point out an additional simplification. By working with geodesic polar coordinates one can first make a small compactly supported perturbation of the  original metric  in $F^{-1}(l_{\alpha})$ near $q$,   and assume in ($\dagger$) that the curvature is a negative constant in a neighborhood of $q$.

To continue with the proof, it  follows from  ($\dagger$)  and the definition of $\nu$  that the $\tilde g$-geodesic that passes through $p_1$ and $p_2$ does not contain $p_3$. Since $\tilde g$ is an arbitrarily small perturbation of   the metric induced by $g$,  and  these metrics coincide off $B_{\epsilon}(q)$,  it is possible to arrange matters so that there is a small perturbation $\hat g$ of the ambient metric $g$ that  coincides  with $g$ off  $B_{\epsilon}(q)$ and  induces $\tilde g$ on
$F^{-1}(l_{\alpha})$.

Since there are only finitely many exceptional lines $l_{\alpha}$,  and finitely many possibilities for three of the points  in $\{p_1, p_2, p_3^{\alpha}, p_4^{\alpha}\}$ to lie on the same geodesics of $F^{-1}(l_{\alpha})$, one can choose the points $q$ in the discussion above to be pairwise distinct for the various choices of $l_{\alpha}$,  and $\epsilon$ to be small enough,  so that the corresponding balls $B_{\epsilon}(q)$ are  disjoint.

Hence, by making a finite number of arbitrarily small perturbations of $g$, all supported on arbitrarily small disjoint sets  away from  $F^{-1}(0)$,  we can construct a metric $\hat g$ as in the statement of the lemma by taking $K$ to be the union of the various closed disjoint open balls $\overline {B_{\epsilon}}$.

Notice that $\hat g$ is not K$\ddot{\text a}$hler, but by taking $\epsilon$ small enough the metric induced by $\hat g$ on any $F^{-1}(l)$ has negative curvature.
\end{proof}
\vskip10pt

\begin{prop} \label{homotopy}
 The maps $\Phi$  and  $\Psi$ are homotopic.
\end{prop}
\noindent The proof   of this proposition is divided into three cases which need to be examined separately.
\vskip5pt
\noindent \bf Case $1$\rm: $p_1=p_3$ and $p_2=p_4$.
\\
This is the trivial alternative, since  $\Phi=\Psi$.
\vskip5pt
\noindent \bf Case $2$\rm: $p_1=p_3$ and  $p_2\neq p_4$,  or $p_2=p_4$ and  $p_1\neq p_3$.
\vskip5pt
\noindent
We shall assume that $p_1=p_3$, $p_2\neq p_4$. The proof is similar if
$p_2=p_4$, $p_1\neq p_3$. Let $\rho_l:[0,1]\rightarrow F^{-1}(l)$ be the unique (non-constant) geodesic such that
$\rho_l(0)=p_2, \rho_l(1)=p_4$, as depicted in  figure $1$.

\begin{center}
\begin{pspicture}(0,-3.028125)(5.2428126,3.028125)
\psline[linewidth=0.04cm](0.9009375,2.5296874)(0.9209375,-1.4103125)
\psline[linewidth=0.04cm](0.9009375,2.5496874)(4.8209376,2.5496874)
\psline[linewidth=0.04cm](4.8209376,2.5496874)(0.9209375,-1.3703125)
\usefont{T1}{ptm}{m}{n}
\rput(1.0023438,-1.6603125){$p_1=p_3$}
\usefont{T1}{ptm}{m}{n}
\rput(1.9259375,-2.8003125){Figure 1}
\usefont{T1}{ptm}{m}{n}
\rput(0.6923438,2.8396876){$p_2$}
\usefont{T1}{ptm}{m}{n}
\rput(4.7523437,2.7996874){$p_4$}
\usefont{T1}{ptm}{m}{n}
\rput(0.40234375,0.8196875){$w_l$}
\usefont{T1}{ptm}{m}{n}
\rput(3.2223437,0.4196875){$v_l$}
\usefont{T1}{ptm}{m}{n}
\rput(2.5423439,2.8396876){$\rho_l$}
\usefont{T1}{ptm}{m}{n}
\rput(1.9323437,1.6196876){$T_l$}
\psline[linewidth=0.04cm](1.5809375,2.5496874)(0.9409375,-1.3103125)
\psline[linewidth=0.04cm](2.8609376,2.5296874)(0.9609375,-1.2903125)
\end{pspicture}
\end{center}

 By Lemma \ref{metric}, we can choose $D^2$ in  the definition of $\Psi$ small enough so  that
$p_1$ does not lie in $\rho_l$ for any $l\in \bar{D^2}$.
Let $T_l:[0,1] \times [0,1] \rightarrow F^{-1}(l)$ be given by the  family of geodesics $T_l(s,  .)$ such that $T_l(s,0)=p_1, T_l(s,1)=\rho_l(s)$, where $l\in \partial\bar{D^2}$.

Recall that  $\omega_l:[0,1]\rightarrow F^{-1}(l)$ is the unique geodesic such that $\omega_l(0)=p_1, \omega _l(1)=p_2$,
and  $\upsilon_l:[0,1]\rightarrow F^{-1}(l)$ is the unique geodesic such that $\upsilon_l(0)=p_3, \upsilon _l(1)=p_4.$
Thus,  $s \mapsto T_l(s, .)$ is a homotopy between  $\omega_l$ and $\upsilon_l$.
Since $p_1$ does not lie in $\rho_l$, by Lemma \ref{metric}, we see that
$$T_l'(s,0)=\frac{\partial}{\partial t}T_l(s,t)|_{t=0}\neq 0$$
for any $l\in \partial \bar{D^2}$ and $s\in [0,1]$.
Let $\Theta$ be the composite map
\begin{eqnarray}
\partial \bar{D^2} \times [0,1] \rightarrow SL_{|\bar{D^2}}\simeq \bar{D^2} \times S^1 \rightarrow S^1, \;\;\;
(l, s) \mapsto \frac{\tau(l,s)}{\|\tau (l,s))\|} \mapsto \pi (\frac{\tau(l,s)}{\|\tau (l,s))\|}),
\end{eqnarray}
where
\begin{eqnarray}
\tau: \partial \bar{D^2} \times [0,1] \rightarrow L, \; \;\;(l,s) \mapsto dF(T_l'(s,0)).
\end{eqnarray}
It is now easy to see that the map  $\Theta$ is a homotopy between $\Phi$ and $\Psi$.
\vskip5pt

\noindent \bf Case $3$\rm: $p_1 \neq p_3$ and $p_2\neq p_4$.
\vskip5pt

\noindent By Lemma \ref{metric}, we can choose $D^2$ small enough such that for any $l\in \partial \bar{D^2}$,
no three points of $p_1, p_2, p_3, p_4$ lie in the same geodesic in $F^{-1}(l)$.
Fix $l\in \partial \bar{D^2}$ and  let $\alpha_l(s):[0,1]\rightarrow F^{-1}(l)$  be the unique geodesic such that $\alpha_l(0)=p_1, \alpha_l(1)=p_3$. Likewise, let  $\beta_l(s):[0,1]\rightarrow F^{-1}(l)$ be the unique geodesic such that $\beta_l(0)=p_2, \beta_l(1)=p_4.$

Let $H_l:[0,1] \times [0,1] \rightarrow F^{-1}(l)$ be given by the  family of geodesics joining  $H_l(s,0)=\alpha_l(s)$ to  $H_l(s,1)=\beta_l(s)$.
Notice that,  for each $s\in [0,1]$, $H_l(s,\cdot)$ is uniquely determined by $\alpha_l(s)$ and $\beta_l(s)$, since $F^{-1}(l)$ is a Cartan-Hadamard surface.
Then $H_l$ induces  a homotopy between $\omega_l$ and $\upsilon_l$ (see figure $2$).

\begin{center}
\begin{pspicture}(0,-2.718125)(7.1428127,2.718125)
\psframe[linewidth=0.04,dimen=outer](6.2409377,2.2596874)(0.9209375,-1.0803125)
\psline[linewidth=0.04cm](2.5009375,2.1796875)(2.5009375,-1.0603125)
\psline[linewidth=0.04cm](4.4209375,2.2396874)(4.4409375,-1.1003125)
\usefont{T1}{ptm}{m}{n}
\rput(3.3204687,-2.4903126){Figure 2}
\usefont{T1}{ptm}{m}{n}
\rput(0.7523438,-1.3903126){$p_1$}
\usefont{T1}{ptm}{m}{n}
\rput(0.85234374,2.4896874){$p_2$}
\usefont{T1}{ptm}{m}{n}
\rput(6.3123436,-1.3303125){$p_3$}
\usefont{T1}{ptm}{m}{n}
\rput(6.2523437,2.5296874){$p_4$}
\usefont{T1}{ptm}{m}{n}
\rput(0.40234375,0.8096875){$w_l$}
\usefont{T1}{ptm}{m}{n}
\rput(6.702344,0.8096875){$v_l$}
\usefont{T1}{ptm}{m}{n}
\rput(3.3223438,-1.4503125){$\alpha_l$}
\usefont{T1}{ptm}{m}{n}
\rput(3.4623437,2.5296874){$\beta_l$}
\usefont{T1}{ptm}{m}{n}
\rput(2.9023438,0.7296875){$H_l$}
\end{pspicture}
\end{center}

\begin{lem} \label{closed}
After labeling  $p_3$, $p_4$ suitably,  $\alpha_l$ does not intersect $\beta_l$ for any $l\in \partial \bar{D^2}$.  In particular, for any $l\in \partial \bar{D^2}$ and $s\in [0,1]$,
$
H_l'(s,0)=\frac{\partial}{\partial t}H_l(s,t)|_{t=0}\neq 0.
$

\end{lem}

We need   the following  well-known convexity result  in the geometry of surfaces of non-positive curvature.  For completeness, we include its proof.
\begin{lem} \label{convex}
Let $N^2$ be a Cartan-Hadamard surface and $\delta:(-\infty, +\infty )\rightarrow N^2$  a geodesic in $N^2$. Then each component of $N^2-\delta$ is convex.
\end{lem}

\begin{proof}
Fix any two points $x_1,x_2$ lying in the same component $U$ of $N^2-\delta$.
Let $\gamma:[0,1]\rightarrow N^2$ be the unique geodesic such that $\gamma(0)=x_1$ and $\gamma(1)=x_2$, as shown in figure $3$.

\begin{center}
\begin{pspicture}(0,-1.7170312)(7.2228127,2.7392187)
\psline[linewidth=0.04cm](0.0609375,2.2207813)(6.9209375,2.3007812)
\usefont{T1}{ptm}{m}{n}
\rput(6.742344,-0.30921876){$x_1$}
\usefont{T1}{ptm}{m}{n}
\rput(0.7423437,-0.24921875){$x_2$}
\usefont{T1}{ptm}{m}{n}
\rput(3.4929688,-1.4892187){Figure 3}
\usefont{T1}{ptm}{m}{n}
\rput(0.448125,1.2507813){U}
\usefont{T1}{ptm}{m}{n}
\rput(0.53234375,2.5507812){$\delta$}
\usefont{T1}{ptm}{m}{n}
\rput(4.2123437,1.5107813){$\gamma_{t_0}$}
\usefont{T1}{ptm}{m}{n}
\rput(3.3823438,-0.36921874){$\zeta$}
\usefont{T1}{ptm}{m}{n}
\rput(6.2123437,1.6507813){$\gamma$}
\psarc[linewidth=0.04](3.7109375,-0.04921875){2.65}{0.0}{180.0}
\psline[linewidth=0.04cm](1.1009375,-0.01921875)(6.3609376,0.0)
\psarc[linewidth=0.04](3.3809376,-0.03921875){2.28}{0.0}{180.0}
\usefont{T1}{ptm}{m}{n}
\rput(5.552344,-0.30921876){$\zeta(t_0)$}
\end{pspicture}
\end{center}

Assume, by contradiction, that  $\gamma$  intersects $\delta$. Take any  continuous curve $\zeta:[0,1]\rightarrow N^2$
such that $\zeta(0)=x_1, \zeta(1)=x_2$,  and $\zeta$ does \it not \rm intersect  $\delta$.
Consider
$$t_0=\sup\{{t|\gamma_t \cap \delta} \neq \emptyset , t\in [0,1]\},$$
 where $\gamma_t: [0,1]\rightarrow N^2$ is the unique geodesic such that
$\gamma_t(0)= \zeta(t)$ and $\gamma_t(1)=x_2$ for any $t \in [0,1]$. Notice that $\gamma_0=\gamma$.
Since $\zeta$ does not intersect  $\delta$, we see that $t_0<1$.
Then $\gamma_{t_0}$ must be tangent to $\delta$, otherwise  $\gamma_t$ will also intersect  $\delta$
for some $t>t_0$, which contradicts  the definition of $t_0$. Hence $\gamma_{t_0}$ must be tangent to $\delta$, and so these curves must coincide as they are geodesics.  Since $x_2 \in \gamma_{t_0}$,  we  have  $x_2 \in \delta$,  a contradiction.
\end{proof}

Now,  we are going to prove Lemma \ref{closed}.
We start by fixing $l_0 \in \partial \bar{D^2}$. We first show that,  by relabelling $p_3$ and  $p_4$ if necessary,  $\alpha_{l_0}$ does not intersect $\beta_{l_0}.$
Suppose $\alpha_{l_0} \cap  \beta_{l_0} \neq \emptyset$.
By Lemma \ref{metric}, $p_2, p_4$ do not lie in $\alpha_{l_0}$. Since $\alpha_{l_0} \cap  \beta_{l_0} \neq \emptyset$, Lemma \ref{convex}
implies that $p_2, p_4$ must lie in different components of $N^2-\alpha_{l_0}$ (see figure $4$).

\begin{center}
\begin{pspicture}(0,-3.758125)(10.642813,3.758125)
\psline[linewidth=0.04cm](1.9209375,3.0796876)(8.380938,-2.1803124)
\usefont{T1}{ptm}{m}{n}
\rput(4.9617186,-3.5303125){Figure 4}
\usefont{T1}{ptm}{m}{n}
\rput(2.1923437,-2.8703125){$p_1$}
\usefont{T1}{ptm}{m}{n}
\rput(8.3923435,-2.4703126){$p_2$}
\usefont{T1}{ptm}{m}{n}
\rput(8.9323435,3.5696876){$p_3=\hat{p_4}$}
\usefont{T1}{ptm}{m}{n}
\rput(2.3523438,3.3896875){$p_4=\hat{p_3}$}
\usefont{T1}{ptm}{m}{n}
\rput(5.9823437,1.8096875){$\alpha_{l_0}$}
\usefont{T1}{ptm}{m}{n}
\rput(5.702344,-0.4903125){$\beta_{l_0}$}
\usefont{T1}{ptm}{m}{n}
\rput(1.5623437,0.1496875){$\hat {\alpha_{l_0}}$}
\usefont{T1}{ptm}{m}{n}
\rput(9.082344,0.6896875){$\hat {\beta_{l_0}}$}
\psline[linewidth=0.04cm](1.9409375,3.0796876)(2.2009375,-2.5203125)
\psline[linewidth=0.04cm](2.2209375,-2.5003126)(8.260938,3.3196876)
\psline[linewidth=0.04cm](8.260938,3.2796874)(8.400937,-2.1803124)
\end{pspicture}
\end{center}

In this case,  we  interchange the labeling of $p_3$ and $p_4$. In other words, set  $\hat{p_1}=p_1$, $\hat{p_2}=p_2$,  $\hat{p_3}=p_4$, $\hat{p_4}=p_3$ (figure 4). Also let $\hat {\alpha_{l_0}}:[0,1]\rightarrow F^{-1}(l)$ be the unique geodesic such that
$\hat {\alpha_{l_0}}(0)=p_1$, $\hat {\alpha_{l_0}}(1)=\hat{p_3}$ and $\hat \beta_{l_0}:[0,1]\rightarrow F^{-1}(l)$ be the unique geodesic
such that $\hat {\beta_{l_0}}(0)=p_2$, $\hat {\beta_{l_0}}(1)=\hat{p_4}$.
By Lemma \ref{convex}, $\hat {\alpha_{l_0}}$ and $\hat {\beta_{l_0}}$ lie in
different components of $N^2-\alpha_{l_0}$,  and hence do not intersect  each other.

\par  As we shall see below, the crucial point of the proof  that $\alpha_l$ does not intersect $\beta_l$ for \it any \rm $l\in \partial \bar{D^2}$ is the fact that,  as it was arranged above, after a possible relabeling of $p_3$ and $p_4$,  the conclusion holds for a single line $l_0 \in \partial \bar{D^2}$. We now proceed to show that, in fact, $\alpha_l\cap\beta_l=\emptyset$ for all $l\in\partial \bar{D^2}$.

By continuity,
$\alpha_l$ does not intersect  $\beta_l$ if $l$ is  sufficiently close  to $l_0$. Suppose  that, for some $l$,
$\alpha_l$ does indeed intersect  $\beta_l$. We are going to derive a contradiction.

Let $\ell$ be the complex line such that
$\alpha_\ell$ intersects  $\beta_\ell$  for the  the  ``first time", as a variable line in $\partial \bar{D^2}$ that started at $l_0$ moves towards $\ell$ (say, in the positive sense in $\partial \bar{D^2}$).

Considering  what happens when the first contact occurs, one sees that  one of the possibilities below must occur (see figures $5$, $6$ and $7$):
\vskip7pt
\noindent i)  $\alpha_{\ell}$ is tangent to $\beta_{\ell}$
at some interior point of $\beta_{\ell}$.
\vskip7pt
\noindent ii) $\alpha_ \ell \cap \beta_{\ell}=\{\beta_{\ell}(0)\}=\{p_2\}$.
\vskip7pt
\noindent iii)  $\alpha_{\ell} \cap \beta_{\ell}= \{\beta_{\ell}(1)\}=\{p_4\}$.

\begin{center}
\begin{pspicture}(0,-1.3570312)(9.042812,3.6192188)
\psarc[linewidth=0.04](4.4709377,-0.32921875){3.25}{0.0}{180.0}
\psline[linewidth=0.04cm](1.1809375,2.8407812)(7.6009374,3.0407813)
\psline[linewidth=0.04cm](1.2009375,2.8607812)(1.2209375,-0.21921875)
\psline[linewidth=0.04cm](7.6009374,3.0207813)(7.7209377,-0.07921875)
\usefont{T1}{ptm}{m}{n}
\rput(1.3523438,-0.56921875){$p_1$}
\usefont{T1}{ptm}{m}{n}
\rput(7.952344,-0.5492188){$p_3$}
\usefont{T1}{ptm}{m}{n}
\rput(1.1923437,3.2307813){$p_2$}
\usefont{T1}{ptm}{m}{n}
\rput(7.592344,3.4307814){$p_4$}
\usefont{T1}{ptm}{m}{n}
\rput(4.5223436,3.2707813){$\beta_l$}
\usefont{T1}{ptm}{m}{n}
\rput(6.162344,2.1707811){$\alpha_l$}
\usefont{T1}{ptm}{m}{n}
\rput(0.79234374,1.8107812){$\omega_l$}
\usefont{T1}{ptm}{m}{n}
\rput(8.132343,1.7707813){$\upsilon_l$}
\usefont{T1}{ptm}{m}{n}
\rput(4.1960936,-1.1292187){Figure 5}
\end{pspicture}
\end{center}

\begin{center}
\begin{pspicture}(0,-2.938125)(8.302813,2.938125)
\usefont{T1}{ptm}{m}{n}
\rput(1.2923437,-1.6903125){$p_1$}
\usefont{T1}{ptm}{m}{n}
\rput(7.8123436,-1.6903125){$p_3$}
\usefont{T1}{ptm}{m}{n}
\rput(1.4723438,2.7496874){$p_2$}
\usefont{T1}{ptm}{m}{n}
\rput(7.632344,2.7496874){$p_4$}
\usefont{T1}{ptm}{m}{n}
\rput(4.1023436,2.7096875){$\beta_l$}
\usefont{T1}{ptm}{m}{n}
\rput(0.70234376,1.2896875){$\alpha_l$}
\usefont{T1}{ptm}{m}{n}
\rput(2.2723436,0.2496875){$\omega_l$}
\usefont{T1}{ptm}{m}{n}
\rput(4.981094,-2.7103126){Figure 6}
\psline[linewidth=0.04cm](1.3809375,2.3396876)(7.8209376,2.3596876)
\psline[linewidth=0.04cm](1.4009376,2.2996874)(7.7609377,-1.4003125)
\psellipse[linewidth=0.04,dimen=outer](1.3809375,0.4796875)(0.46,1.84)
\usefont{T1}{ptm}{m}{n}
\rput(5.4823437,0.1896875){$\alpha_l$}
\end{pspicture}
\end{center}

\begin{center}
\begin{pspicture}(0,-2.668125)(8.802813,2.668125)
\usefont{T1}{ptm}{m}{n}
\rput(0.45234376,-1.6603125){$p_1$}
\usefont{T1}{ptm}{m}{n}
\rput(7.2923436,-1.6603125){$p_3$}
\usefont{T1}{ptm}{m}{n}
\rput(0.41234374,2.3396876){$p_2$}
\usefont{T1}{ptm}{m}{n}
\rput(6.972344,2.4796875){$p_4$}
\usefont{T1}{ptm}{m}{n}
\rput(3.7623436,2.3596876){$\beta_l$}
\usefont{T1}{ptm}{m}{n}
\rput(3.1023438,0.2996875){$\alpha_l$}
\usefont{T1}{ptm}{m}{n}
\rput(6.1523438,0.4196875){$\upsilon_l$}
\usefont{T1}{ptm}{m}{n}
\rput(3.72,-2.4403124){Figure 7}
\psline[linewidth=0.04cm](0.8009375,-1.1903125)(7.0009375,2.1096876)
\psline[linewidth=0.04cm](0.2209375,1.9496875)(7.0409374,2.1296875)
\psellipse[linewidth=0.04,dimen=outer](7.0009375,0.4096875)(0.6,1.7)
\usefont{T1}{ptm}{m}{n}
\rput(8.022344,0.8796875){$\alpha_l$}
\end{pspicture}
\end{center}

In alternative i),  $\alpha_{\ell}$ is tangent to $\beta_{\ell}$,  and so these curves  must coincide since they  are geodesics.
Then $p_1, p_2, p_3, p_4$ lie in the same geodesic in $F^{-1}(\ell)$, contradicting  Lemma \ref{metric}.

If  ii) holds,   $\alpha_{\ell} \cap \beta_{\ell}=\beta_{\ell}(0)=p_2$,  and so
$\alpha_{\ell}$ passes through both $p_1$ and $p_2$.  Recall that $\omega_{\ell}:[0,1]\rightarrow F^{-1}(\ell)$ is the unique geodesic
such that $\omega_{\ell}=p_1$ and $\omega_{\ell}(1)=p_2$. Since $F^{-1}(\ell)$ is a Hadamard surface,
we see that $\alpha_{\ell}$ must coincide with $\omega_{\ell}$. Since $\alpha_{\ell}(1)=p_3$, we see that
$p_1, p_2, p_3$ lie in the same geodesic $\omega_{\ell}$, which is again a contradiction to Lemma \ref{metric}. For the same reason,  
iii) cannot happen either.

Hence,  $\alpha_l$ does not intersect $\beta_l$ for any $l\in \partial \bar{D^2}$, thus establishing the first half of Lemma \ref{closed}.

Recall that $H_l(s,t):[0,1] \times [0,1] \rightarrow F^{-1}(l)$ provides  a family of geodesics $H_l(s,.)$ such that
\begin{eqnarray*} H_l(s,0)=\alpha_l(s), \;\; H_l(s,1)=\beta_l(s).\end{eqnarray*}
Since $\alpha_l(s)\neq \beta_l(s)$ for all $s\in [0,1]$, we see that the connecting geodesic is non-constant, and so
\begin{eqnarray}
H_l'(s,0)=\frac{\partial}{\partial t}H(s,t)|_{t=0}\neq 0
\end{eqnarray}
for any $l\in \partial \bar{D^2}$ and $s\in [0,1]$.  This concludes  the proof of Lemma \ref{closed}.\qed
\\
\par
We can now complete the proof of Proposition \ref{homotopy}. Let $\Lambda$ be the composite map
\begin{eqnarray}
\partial \bar{D^2} \times [0,1] \rightarrow SL_{|\bar{D^2}}\simeq \bar{D^2} \times S^1 \rightarrow S^1, \;
(l, s) \mapsto \frac{\varsigma(l,s)}{\|\varsigma(l,s)\|} \mapsto \pi (\frac{\varsigma(l,s)}{\|\varsigma(l,s)\|}),
\end{eqnarray}
where
\begin{eqnarray}
\varsigma: \partial \bar{D^2} \times [0,1] \rightarrow L, \; (l,s) \mapsto dF(H_l'(s,0)).
\end{eqnarray}
It follows from the properties of $H$ that  $\Lambda$ is a homotopy between $\Phi$ and $\Psi$.
This completes  the discussion of Case 3, thus proving  Proposition \ref{homotopy}. \qed
\\
\par
We can now finish  the proof of Theorem \ref{est}.   For each exceptional line $l_{\alpha}$ we constructed  maps $\Phi, \Psi: \partial \bar{D^2}\to S^1$. By Proposition
\ref{homotopy},    $\text{deg}(\Phi)=\text{deg}(\Psi)$. Since $\Psi$ extends to  a continuous map $\bar{\Psi}:\bar{D^2} \rightarrow S^1$,
$\text{deg}(\Psi)=0$  (\cite{His}, p. 126).

Hence $\text{deg}(\Phi)=0$  and therefore, as it was explained  prior to the definition of $\Phi$,    the local index of $s$ at $l_\alpha$ is zero.
Summing over $\alpha$,  and using  the  Poincar\'e-Hopf index theorem (\cite{Bot}, p. 123-124),  the Euler number of the tautological line bundle $L$  would be  zero, a contradiction to $e(L)=-1$.

\vskip10pt
\section{Examples.}
\vskip10pt

 \noindent  A)  It is essential in   Theorem \ref{line}  that   the pre-image  of  each  exceptional line should  have only finitely many components (hypothesis c)), otherwise the fiber may be infinite.

\vskip10pt
In order to see this, consider the simple local   biholomorphism $F:\mathbb C^2 \to \mathbb C^2$, given by
$$F(z_1, z_2)= (e^{z_1}-1, z_2),$$
 so that $\#F^{-1}(0,0)=\infty$.   If $\lambda \in \mathbb C$,   the pre-image of the  line $w_2=\lambda w_1$  is  a graph over $\mathbb C$, hence $1$-connected.   In particular, a) and b) are satisfied.

 On the other hand,
the pre-image of the line $\omega_1=0$ is $\bigcup_{j=1}^{\infty} (\{2\pi j\sqrt{-1}\}\times \mathbb C)$. Hence, although there is only one exceptional line, its pre-image    has infinitely many  components.

\vskip10pt

\noindent B)  The estimate in  Theorem \ref{line} is sharp.

\vskip10pt
We begin with   a locally univalent  entire function $f:\mathbb C\to \mathbb C$ that covers  $0$ exactly $k$ times, $1\leq k <\infty$. These functions can be constructed as follows (we  are grateful to  A.  Weitsman for kindly explaining to us  the case $k=1$).

Let
\begin{eqnarray*} P(z)= \prod_{j=1}^k (z-j), \end{eqnarray*}  so that
\begin{eqnarray*} P'(j)= \prod _{ 1\leq n \leq k, n\neq j} (j-n)\neq 0. \end{eqnarray*}
Let  $\psi$ be any  entire function that solves the finite interpolation problem
$$\psi(j)=\log P'(j), \;\; j=1, \dots, k,$$ for some branch of the logarithm; for instance, it is easy to see  that there is a  (unique) solution $\psi$ that is a polynomial  of degree
$\leq k-1$.

Set
\begin{eqnarray*} \phi(z)=(P(z))^{-1}(e^{\psi(z)}-P'(z)), \end{eqnarray*}  and observe that $z=j$ is a removable singularity for $\phi$, so that $\phi$  extends to an entire function.  Let $\Phi$ be an entire function satisfying  $\Phi'=\phi$  and define
$$f(z)=P(z) e^{\Phi(z)}.$$
Hence, $f$ vanishes precisely at $1, \dots, k$, with multiplicity one at all these points.
Furthermore, $f$ is locally univalent, since
\begin{eqnarray*} f'(z)=(P(z)\phi(z)+P'(z)) e^{\Phi(z)}=e^{\psi(z)}e^{\Phi(z)}\neq 0.\end{eqnarray*}

Consider now the map $F:\mathbb C^2 \to \mathbb C^2$  given by
$F(z_1, z_2)= (f(z_1), z_2)$.
The Jacobian determinant of $F$ at $(z_1, z_2)$ is $f'(z_1)$, hence non-zero. The fiber of $(0,0)$ has precisely $k$ elements, namely $(1,0),  \dots, (k,0)$.

As in Example A),  if $\lambda \in \mathbb C$ the pre-image of the  line $w_2=\lambda w_1$ is  a graph over $\mathbb C$, hence $1$-connected.
On the other hand, if $l$ is the line $w_1=0$, $F^{-1}(l)$  is homeomorphic to
\begin{eqnarray*} f^{-1}(0)\times  \mathbb C=\bigcup_{j=1}^k (\{j\}\times \mathbb C), \end{eqnarray*}
which has  precisely $k$ connected components, the same number of points in  the fiber $F^{-1}(0,0)$. This shows that equality can be achieved in (\ref{sharp}).

\vskip10pt

\noindent \bf  \rm C) The topological hypothesis b), to the effect that that the pre-images of the generic lines passing through $q$ should be   $1$-connected,   is essential for the validity  of Theorem \ref{est}.
\vskip10pt

In order to produce an example illustrating this point, we  need the result   that the  product  of K$\ddot{\text a}$hler manifolds  of negative holomorphic sectional curvature  also has these properties. Since we were unable to  locate  a reference for  the part of this  statement concerning curvature, a proof is included in an appendix.

Every metric on an orientable real surface is K$\ddot{\text a}$hler, relative to the natural complex structure, and so one can endow
$\mathbb C-\{1/2\}$ with a complete conformal  metric  of negative curvature (for instance, by considering  the Weierstrass representation of a catenoid).

Taking  the corresponding  product metric on   $\Omega:=(\mathbb C-\{1/2\})\times  (\mathbb C-\{1/2\})$,
and using  the result in the appendix,  one obtains  a complete K$\ddot{\text a}$hler  metric  of negative holomorphic sectional curvature, as  required  in the statement of Theorem \ref{line}.

Alternatively,  and in a less elementary way, such a metric exists because the Behnke-Stein theorem (\cite{N}, p. 240) applies to show that  $\Omega$, being the product of two open Riemann surfaces,  is also Stein. In particular, $\Omega$  can be given the metric induced from some proper embedding $\Omega \to \mathbb C^N$, where $\mathbb C^N$ is endowed with  a complete K$\ddot{\text a}$hler metric of negative holomorphic sectional curvature \cite{S}.

Next, consider the map
\begin{eqnarray*}G:\mathbb C^2\to \mathbb C^2,
\;\;\; G(z_1, z_2)=(z_1^2-z_1, z_2^2-z_2). \end{eqnarray*}
The  restriction  $F=G|_{\Omega}:\Omega \to \mathbb C^2$ is a local biholomorphism,  hypothesis a) is satisfied, and the fiber
$F^{-1}(0,0)$ consists of  the  four points $(0,0), (0,1), (1,0), (1,1)$.  Giving coordinates
$(w_1, w_2)$  to $\mathbb C^2$ (the target of $F$), one sees that the complex lines $l$ through $(0,0)$ are of the following two types: $w_1=0$,  or  $w_2=k w_1$ for an arbitrary  $k\in \mathbb C$.

 We now analyze when   $F^{-1}(l)$ is connected,  for  $l$  of the second kind.
Since  $G^{-1}(l)$  is given by the zero set of the complex polynomial
\begin{eqnarray*}P_k(z_1, z_2)=z_2^2-kz_1^2 +kz_1-z_2, \end{eqnarray*} and $F^{-1}(l)$ is obtained from $G^{-1}(l)$ by deleting   the  points corresponding to $z_1=1/2$ or $z_2=1/2 $,  one sees that $F^{-1}(l)$ is not simply-connected.

In  order to show that that the pre-image $F^{-1}(l)$ of a generic line  is connected,  it suffices to  show that  the same holds for $G^{-1}(l)$, since the former is obtained from the latter by deleting finitely many points.

Therefore, one needs to know  the  values of $k$  for which the polynomial  $P_k(z_1, z_2)$  is irreducible. The following arguments are standard in elementary algebraic geometry. If $P_k$ is  reducible, the same is  true of  its homogeneization
\begin{eqnarray*} Q_k(z_1, z_2, Z)=Z^2 P_k(z_1/Z, z_2/Z)=
z_2^2-kz_1^2+kz_1Z-z_2Z. \end{eqnarray*}
The usual  criterion for the irreducibility of conics in the projective plane, in terms of  the non-vanishing of the determinant of the associated quadratic form (\cite{Sc}, p. 212), now applies to show that  $Q_k$ is reducible if and only if  $k=0, 1$.

For $k=0$,
$$ F^{-1}\{w_2=0\}=( \{z_2=0\}-\{(1/2,0)\})\bigcup (\{z_2=1\}-\{(1/2,1)\}).$$

For $k=1$, since $P_1(z_1, z_2)= (z_2-z_1)(z_2+z_1-1)$,
$$F^{-1}\{w_2=w_1\}= (\{z_2=z_1\}-\{(1/2, 1/2)\})\bigcup (\{z_2+z_1=1\}-\{(1/2, 1/2)\}).$$

The  pre-image  of the line $\{w_1=0\}$  is

$$ F^{-1}\{w_1=0\}= (\{z_1=0\}-\{(0,1/2)\})\bigcup (\{z_1=1\}-\{(1,1/2)\}).$$
Hence,  there are  three   lines   with disconnected pre-images, namely:   $w_1=0, w_1=w_2$, and $w_2=0$. For each one of them  the pre-image has two connected components (none of which is simply-connected, but this is not required in  hypothesis c) of Theorem \ref{line}). For any other line (i.e., $w_2=kw_1, \; k\neq 0, 1 $), the pre-image is connected but not simply-connected.

Thus,  the number of  connected components of the pre-image of each of the  three  lines  with disconnected pre-image  is two, which  is strictly less than  four, the  number of points in   $F^{-1}(0,0)$.  Hence,  the  estimate (\ref{sharp'})  does not hold.

This is compatible with Theorem \ref{line}. Indeed, although a) and c) are satisfied (in fact, the pre-image of any complex line has only finitely many components),   the pre-image of the line $w_2=k w_1$, $k\neq 0, 1$,   is    connected but   not simply-connected. Since the complement of any finite set  in $\mathbb C \mathbb P^1$ would have to contain points corresponding to such  lines, regardless of what Zariski open set one considers  there  will be  generic lines whose pre-images  are not simply-connected, and so  hypothesis b)  does not hold.

\vskip10pt
\section{Appendix.}
\vskip10pt

\begin{lem}\label {product}
Let $(M_i, J_i, g_i)$ be  K$\ddot{ a}$hler manifolds  of negative  holomorphic sectional curvature,
$j=1,2$.  Then $$(M_1 \times M_2, J_1 \oplus J_2, g_1 \times g_2)$$ is also  a K$\ddot{ a}$hler manifold of negative  holomorphic sectional curvature.
\end{lem}

It is standard that the product manifold is also K$\ddot{\text a}$hler.  In order to show that  $g_1 \times g_2$ has negative holomorphic sectional curvature, let $\nabla$, $\nabla^1$, $\nabla^2$ be the Levi-Civita connections
associated to $g_1 \times g_2$, $g_1$, $g_2$, respectively. For $X_1, X_2 \in \Gamma(TM_1)$, $Y_1, Y_2 \in \Gamma(TM_2)$, we have,  for  $i, j=1,2$,
\begin{eqnarray} \label{China}\nabla_{X_i}Y_j=0, \;\; \nabla_{Y_i}X_j=0, \;\; [X_i, Y_j]=0,\;\;\nabla_{X_i}X_j=\nabla^1_{X_i}X_j, \;\; \nabla_{Y_i}Y_j=\nabla^2_{Y_i}Y_j. \end{eqnarray}
\par For any $Z \in  \Gamma(TM_1 \times TM_2)$, $Z \neq 0$,
we can write $Z= X+Y$, where $X\in \Gamma(TM_1)$ and  $Y \in \Gamma(TM_2)$. Let $R, R_1, R_2$ be the
Riemannian curvature tensors associated to $g_1 \times g_2$, $g_1$, $g_2$, respectively. Letting  $J=J_1 \oplus J_2$, one has
$$R(Z, JZ, Z, JZ)=R(X+Y, J X+J Y, X+Y, J X + J Y)$$
$$=R(X, JX, X, JX)+R(X, JX, X, JY)+R(X, JX, Y, JX)+R(X, JX,Y, JY)$$
$$+R(X, JY, X, JX)+R(X, JY, X, JY)+R(X, JY, Y, JX)+R(X, JY, Y, JY)$$
$$+R(Y, JX, X, JX)+R(Y, JX, X, JY)+R(Y, JX, Y, JX)+R(Y,JX,Y,JY)$$
$$+R(Y,JY,X,JX)+R(Y,JY,X,JY)+R(Y,JY,Y,JX)+R(Y,JY,Y,JY).$$

Using (\ref{China})  and   the definition of the curvature tensor,  one can show that  among  the sixteen summands above  only the first and the last terms  can possibly be  non-zero. Thus,

$$R(Z, JZ, Z, JZ)=R(X,JX,X,JX)+R(Y,JY,Y,JY)$$
$$=R_1(X, J_1 X, X, J_1 X)+ R_2(Y, J_2Y, Y, J_2 Y)<0, $$
as $Z\neq 0$  implies  $X\neq 0$ or $Y\neq 0$, and  the holomorphic sectional curvatures are  negative.

The verification that the intermediate terms vanish is  straightforward. We illustrate this by computing  the second,  fourth, and  seventh  terms:
$$R(X, JX, X, JY)=g(\nabla_X \nabla_{JX}X- \nabla_{JX} \nabla_{X}X- \nabla_{[X,JX]}X, JY)$$

$$=g(\nabla^1_X \nabla^1_{JX}X- \nabla^1_{JX} \nabla^1_{X}X- \nabla^1_{[X,JX]}X, JY)=0,$$
since $JY$ is orthogonal to  $TM_1$.

Similarly,
$$R(X,JX,Y,JY)=g(\nabla_X \nabla_{JX}Y- \nabla_{JX}\nabla_X Y- \nabla_{[X,JX]}Y,JY)=g(0,JY)=0, $$
and
$$R(X,JY,Y,JX)=g(\nabla_X \nabla_{JY}Y-\nabla_{JY}\nabla_XY-\nabla_{[X,JY]}Y, JX)=g(0,JX)=0.$$

\vskip10pt

$$
\begin{array}{lcccccccccl}
\text{Xiaoyang Chen}            &&&&&&&&& & \text{Frederico Xavier}\\
\text{Department of Mathematics}      &&&&&&&&& & \text{Department of Mathematics}\\
\text{University of Notre Dame}&&&&&&&&& & \text{University of Notre Dame}\\
\text{Notre Dame, Indiana, 46556, USA}          &&&&&&&&& & \text{Notre Dame, Indiana, 46556, USA}\\
\text{xchen3@nd.edu}           &&&&&&&&& & \text{fxavier@nd.edu}\\
\end{array}
$$


\end{document}